\documentclass[11pt]{amsart}
\usepackage{amssymb}
\usepackage{amsmath}
\usepackage[active]{srcltx}
\usepackage[utf8]{inputenc}
\usepackage{verbatim}
\usepackage{amsmath,amsfonts,amssymb,amsthm}
\usepackage[mathcal]{eucal}
\usepackage{enumerate}
\usepackage[centertags]{amsmath}
\usepackage{graphics}
\usepackage{color}

\newtheorem{thm}{Theorem}
\newtheorem{cor}{Corollary}
\newtheorem{rem}{Remark}

\newtheorem{theorem}[thm]{Theorem}

\newtheorem{corollary}[cor]{Corollary}

\newtheorem{remark}[rem]{Remark}

\begin{document}
\author{N. Areshidze, L.-E. Persson and G. Tephnadze}
\title[F\'ejer means]{On the divergence of F\'ejer means with respect to Vilenkin systems on the set of measure zero}
\address {N. Areshidze, Department of Mathematics, Faculty of Exact and Natural	Sciences, Tbilisi State University, Chavchavadze str. 1, Tbilisi 0128, Georgia}
\email{nika.areshidze15@gmail.com}
\address{L.-E. Persson, UiT The Arctic University of Norway, P.O. Box 385, N-8505, Narvik, Norway and Department of Mathematics and Computer Science, Karlstad University, 65188 Karlstad, Sweden.}
\email{larserik6pers@gmail.com}
\address{G. Tephnadze, The University of Georgia, School of science and technology, 77a Merab Kostava St, Tbilisi 0128, Georgia.}
\email{g.tephnadze@ug.edu.ge}

\thanks{The research was supported by Shota Rustaveli National Science Foundation grant no. FR-21-2844.}
\date{}

\begin{abstract}
The famous Carleson-Hunt theorem has been in focus of interest for a long time.  This theorem concerns convergence almost everywhere of Fourier series of $f\in L_p$ functions for $1<p\leq \infty.$  Kolmogorov  constructed a function $f\in L_1$ such that the partial sums of Fourier series  diverge everywhere. On the other hand, we have boundedness result for F\'ejer means for all $1\leq p\leq \infty$. Similar results are proved for the partial sums and  F\'ejer means of Vilenkin-Fourier series.
But also here it appears the questions what happens on any subset $E$ of measure zero, can we even have a function which diverge there? We contribute with a new result concerning this question and prove by the concrete construction that for any set $E$ of measure zero there exists a function $f\in L_p(G_m) (1\leq p<\infty)$ such that the F\'ejer means with respect to Vilenkin systems diverge on this set, which follows similar result for the partial sums. The key is to use new constructions of Vilenkin polynomials, which was introduced in \cite{PTW2}. 
In fact, the theorem we prove follows from the general result of \cite{Kar}, 
but we provide an alternative approach and the constructed function in our proof has a simple explicit representation.
\end{abstract}

\maketitle

	

\section{Introduction including some historical results}

The almost-everywhere convergence of Fourier series for $L_2$ functions was postulated by Luzin \cite{Luzin} in 1915 and the problem was known as Luzin's conjecture. Carleson's theorem is a fundamental result in mathematical analysis establishing the pointwise (Lebesgue) almost everywhere convergence of Fourier series of $L_2$ functions, proved by Carleson \cite{Carleson} in 1966. The name is also often used to refer to the extension of the result by Hunt \cite{Hunt}, which was given in 1968 to $L_p$ functions for $p\in(1, \infty) $ (also known as the Carleson-Hunt theorem). Carleson's original proof is exceptionally hard to read, and although several authors have simplified the arguments there are still no easy proofs of his theorem. Fefferman \cite{Feferman} published a new proof of Hunt's extension, which was done by bounding a maximal operator of partial sums
$
S^{\ast}f:=\sup_{n\in\mathbb{N}}\left\vert S_nf\right\vert.
$
This, in its turn, inspired to derive a much simplified proof of the $L_2$ result by Lacey and Thiele \cite{LT}, explained in more detail in Lacey \cite{Lac}. In the books Fremlin \cite{Fre} and Grafakos \cite{Graf} were also given proofs of Carleson's theorem. 
Already in 1923, Kolmogorov \cite{Kol} showed that the analogue of Carleson's result for $L_1$ is false by finding such a function whose Fourier series diverges almost everywhere (improved slightly in 1926 to diverging everywhere). 

The analogue of Carleson's theorem for the Walsh system was proved by Billard \cite{Billard1967} for $p=2$ and by Sjölin \cite{sj1} for $1 <p<\infty$, and for bounded Vilenkin systems by
Gosselin \cite{goles}. Schipp \cite{s1} investigated the so called tree martingales generalized the results about the maximal function, quadratic variation and martingale transforms to these martingales. Using these results, he gave a proof of Carleson's theorem for Walsh-Fourier series. A similar proof for bounded Vilenkin systems can be found in Schipp and Weisz \cite{s3} (see also \cite{wk}). A proof of almost everywhere convergence of Walsh-Fourier series was also given by Demeter \cite{demeter007} in 2015. Moreover, the corresponding result for Vilenkin-Fourier series was proved in \cite{PTW2} (see also \cite{PSTW}) in 2022. 
Stein \cite{steindiv} constructed an integrable function whose Walsh-Fourier series diverges almost everywhere. Later Schipp \cite{Schipp1969a} (see also \cite{sws}) proved that there exists an integrable function whose Walsh-Fourier series diverges everywhere. Kheladze \cite{Khe1,Khe2} proved that for any set of measure zero there exists a function in $f\in L_p(G_m)$ $(1<p<\infty)$ whose Vilenkin-Fourier series diverges on the set,  while
the corresponding result for continuous or bounded function was proved by  Bitsadze \cite{Bitsadze}, Goginava \cite{Gog} and Harris \cite{Harris1986}.  Another resent proof was given in the new book \cite{PTW2}.

In the one-dimensional case the weak type inequality for the maximal operator of maximal operator of Fej\'{e}r means 
\begin{equation*}
\lambda \mu \left( \sigma ^{*}f>\lambda \right) \leq c\left\|
f\right\|_{1}\text{ \quad }\left( \lambda >0\right)
\end{equation*}
can be found in Zygmund \cite{Zy} for trigonometric series, in Schipp
\cite{Sc} for Walsh series and in Pál, Simon \cite{PS} for bounded Vilenkin
series.
It follows (see \cite{NPTW} and \cite{PTW2}) that the Fej\'{e}r means $\sigma_nf$ of the Vilenkin-Fourier series of $f\in L_1$ converge to $f$ almost everywhere.

Bugadze \cite{Bug} proved that for any set of measure zero there exists a function in $f\in L_1(G_m)$ whose F\'ejer means of the Walsh-Fourier series diverges on the set. 
Karagulyan \cite{Kar} (see also \cite{Kar1}) considered sequences of linear operators $\{U_nf(x)\}$ with some localization  property, which is fulfilled for all summability methods with respect to orthonormal systems and proved that for any set $E$ of measure zero there exists a set $A$
for which $\{U_nI_A(x)\}$ diverges at each point $x \in E,$ where $I_A$ is a characteristic function  of a subset $A.$ This result is a generalization of analogous theorems known for the Fourier sum operators with respect to different orthogonal systems. 

In this paper we construct an explicit function in $L_p(G_m) (1\leq p<\infty)$  whose F\'ejer means with respect to the Vilenkin system diverges on any set of measure zero.  In fact, the theorem we prove follows from the general result of \cite{Kar}, 
but we provide an alternative approach and the constructed function in our proof has a simple and explicit representation.

In section 3 we present and prove our main result. Some corollaries and final remarks are given in Section 4. In order not to disturb our discussion in Section 1 and 3 we preserve Section 2 for some preliminaries.

\section{Preliminaries}
	
Denote by $\mathbb{N}_{+}$ the set of the positive integers, $\mathbb{N}:=\mathbb{N}_{+}\cup \{0\}.$ 
Denote by
$
Z_{m_{k}}:=\{0,1,...,m_{k}-1\}
$
the additive group of integers modulo $m_{k}$.
	
Define the group $G_{m}$ as the complete direct product of the groups $Z_{m_{i}}$ with the product of the discrete topologies of $Z_{m_{j}}`$s. 
The elements of $G_{m}$ are represented by sequences
$$
x:=\left( x_{0},x_{1},\cdots,x_{j},\cdots\right), \left( x_{j}\in Z_{m_{j}}\right) .
$$

The direct product $\mu $ of the measures
$
\mu_{k}\left( \{j\}\right) :=1/m_{k}\text{ \ }(j\in Z_{m_{k}})
$
is the Haar measure on $G_m$ with $\mu \left( G_{m}\right) =1.$

It is easy to give a base for the neighborhood of $G_{m}:$
\begin{equation*}
I_{0}\left( x\right) :=G_{m},\text{ \ }I_{n}(x):=\{y\in G_{m}\mid y_{0}=x_{0},\cdots,y_{n-1}=x_{n-1}\},
\end{equation*}%
where $x\in G_{m},$ $n\in\mathbb{N}.$
Denote $I_{n}:=I_{n}\left( 0\right) $ for $n\in\mathbb{N}_{+},$ and $\overline{I_{n}}:=G_{m}$ $\backslash $ $I_{n}$.
	
If we define the so-called generalized number system based on $m$ by
\begin{equation*}
M_{0}:=1,\ M_{k+1}:=m_{k}M_{k}\,\,\,\ \ (k\in \mathbb{N}),
\end{equation*}%
then every $n\in\mathbb{N}$ can be uniquely expressed as 
$$n=\sum_{j=0}^{\infty }n_{j}M_{j},$$ 
where $n_{j}\in Z_{m_{j}}$ $(j\in\mathbb{N}_{+})$ and only a finite numbers of $n_{j}`$s differ from zero. For two natural numbers $n=\sum_{j=1}^{\infty}n_jM_j$ and $k=\sum_{j=1}^{\infty}k_jM_j$, we define the operation $\oplus$ by
$$
n \oplus k:=\sum_{i=0}^{\infty} \left(\left(n_j+k_j\right)\pmod {m_i}\right) {M_{j}}, \quad n_j, k_j \in Z_{m_j}.
$$
It is known that (see \cite{PTW2})
\begin{eqnarray}\label{vil}
\psi _{n}\left( x+y\right) &=&\psi _{n}\left( x\right) \text{\ }
\psi _{n}\left( y\right) ,   \\
\psi _{n}\left( -x\right) &=&\psi _{n^{\ast }}\left( x\right) =\overline{\psi }_{n}\left( x\right) ,\text{ \ \ \ }  \notag \\
\psi_{n\oplus k}\left( x\right) &=&\psi_{k}\left( x\right)\psi _{n}\left( x\right) ,\text{\ \ }\left( k,n\in \mathbb{N},\text{ \ \ }x,y\in G_{m}\right).  \notag
\end{eqnarray}

We define the complex-valued function 
$r_{k}\left( x\right) :G_{m}\rightarrow\mathbb{C},$ called the generalized Rademacher functions, by
\begin{equation*}
r_{k}\left( x\right) :=\exp \left( 2\pi \imath x_{k}/m_{k}\right) ,\text{ }\left(\imath^{2}=-1, \ \ x\in G_{m},\text{ \ \ }k\in\mathbb{N}\right) .
\end{equation*}

Now, define the Vilenkin system (see \cite{Vi1})
$\psi :=(\psi _{n}:n\in\mathbb{N})$ on $G_{m}$ as:
\begin{equation*}
\psi _{n}(x):=\prod\limits_{k=0}^{\infty }r_{k}^{n_{k}}\left( x\right),\,\,\ \ \,\left( n\in \mathbb{N}\right).
\end{equation*}
Specifically, we call this system the Walsh-Paley system when $m\equiv 2.$
The Vilenkin system is orthonormal and complete in $L_{2}\left( G_{m}\right)$  (see e.g. the books \cite{AVD} and \cite{sws}).
	
If $f\in L_{1}\left( G_{m}\right) $, we define the Fourier coefficients, partial sums of the Fourier series, Dirichlet kernels with respect to the Vilenkin system by
\begin{equation*}
\widehat{f}\left( n\right) :=\int_{G_{m}}f\overline{\psi }_{n}d\mu,\text{ \ \ }\left( n\in\mathbb{N}\right)
\end{equation*}
\begin{equation*}
S_{n}f:=\sum_{k=0}^{n-1}\widehat{f}\left( k\right) \psi _{k}\text{ \  and \ }
D_{n}:=\sum_{k=0}^{n-1}\psi_{k},\text{ \ \ }\left( n\in\mathbb{N}_{+}\right)
\end{equation*}%
respectively. Recall that (see e.g.  \cite{AVD} and  \cite{PTW2})
\begin{equation}  \label{1dn}
D_{M_{n}}\left( x\right) =\left\{
\begin{array}{ll}
M_{n}, & \text{if\thinspace \thinspace \thinspace } x\in I_{n}, \\
0, & \text{if}\,\,x\notin I_{n}.%
\end{array}%
\right. 
\end{equation}

For any $n\in\mathbb{N}_{+},$ the $n$-th F\'ejer means and kernels with respect to Vilenkin and systems are respectively defined by
\begin{eqnarray*} 
\sigma_{n}f:=\frac{1}{n}\sum_{k=0}^{n-1}S_{k}f=\sum_{k=0}^{n-1}
\left(1-\frac{k}{n}\right)\widehat{f}\left( k\right)\psi_k, \  K_{n}:=\frac{1}{n}\sum_{k=0}^{n-1}D_{k}=\sum_{k=0}^{n-1}
\left(1-\frac{k}{n}\right)\psi_k.
\end{eqnarray*}

We also define the maximal operators $S^{\ast}$ and $\sigma^{\ast}$ of the partial sums and F\'ejer means with respect to Vilenkin and trigonometric systems, respectively by
$$
S^{\ast}f:=\sup_{n\in\mathbb{N}}\left\vert S_nf\right\vert \ \ \ \ \ \text{and}\ \ \ \ \
\sigma^{\ast}f:=\sup_{n\in\mathbb{N}}\left\vert \sigma_nf\right\vert.
$$

\section{Almost everywhere divergence of Vilenkin-Fourier series}

Our main theorem reads:

\begin{theorem}
	\label{Theorem_11}
	Let $1\leq p<\infty$ If $E\subseteq G_m$ is a set of measure zero, then there exists a function $f\in L_p(G_m)$ such that F\'ejer means with respect to Vilenkin systems diverge on this set. 
\end{theorem}

\begin{proof}
	We begin with a general remark. If $A \subseteq G_m$ is a finite union of  Vilenkin intervals $I_1, I_2, \dots , I_n$ for some $n \in \mathbb{N}_+$ and if $N$ is any non-negative integer, then there exists a Vilenkin	polynomial $P$ such that, for some $i\geq N,$
	$$
	P=\sum_{k=M_N}^{M_i-1}c_k\psi_k,
	$$
	which satisfies
	$\left|P(x)\right|=1, \  (x\in A) \  \text{and} \  \int_{G_m}\left|P\right|^p d\mu=\mu (A).$
	Indeed, if $s := \max \{ M_N, 1/\mu (I_j) :1\leq j \leq n \},$ then, in view of \eqref{1dn}, we find that
	$P:= \chi_A\psi_{s}$
	is such a polynomial. In particular,  for the character function $\chi_{_{I_n(y)}}$ of any set $I_n(y)$ it can be written by explicit form $$\chi_{I_n(y)}=D_{M_n}(x-y)/M_n.$$
	According to \eqref{vil} we can conclude that it is Vilenkin polynomial.
	
	To prove the theorem, suppose $E \subseteq G_m$ satisfies that $\mu (E)= 0.$ Cover $E$ with intervals	$(I_k, k\in \mathbb{N})$ such that
	$\sum_{k=0}^{\infty}\mu (I_k)<1$
	and each $x \in E$ belongs to infinitely many of the sets $I_k.$ Set $n_0 := 0$ and choose integers $n_0 < n_1 < n_2 \dots$ such that
	$$\sum_{k=n_j}^{\infty}\mu (I_j)<M_j^{-1} \ \  (j \in \mathbb{N}).$$
	Apply the general remark above successively to the sets
	$$A_j := \bigcup_{k=n_j}^{n_{j+1}-1} I_k \ \ \ \ \ \ (j \in \mathbb{N}) $$	
	to generate integers $\alpha_0 :=0 < \alpha_1 < \alpha_2 < \dots$ and Vilenkin polynomials $P_0, P_1,\dots$ such that
	$sp(P_j):=\{n\in\mathbb{N}: \widehat{P}_j\neq 0 \}\subset \left[M_{\alpha_j}, M_{\alpha_{j+1}} \right):$
	$$P_j=\sum_{k=M_{\alpha_j}}^{M_{\alpha_{j+1}}-1}c_k^j\psi_k,$$
	\begin{equation}
	\label{24} \|P_j \|_p^p=\mu(A_j)\leq M_j^{-1}
	\end{equation}
	and
	\begin{equation}\label{25} 
	\left|P_j(x)\right|=1 \ \ \  x\in A_j, \ \ \ 	\text{for} \ \  \ j \in \mathbb{N}.
	\end{equation}	
	Let
	$$g:=\sum_{j=1}^{\infty}P_j $$		
	and choose integers $\beta_0<\beta_1<\beta_2<\ldots,$ such that $\alpha_{j}<\beta_{j}, \  j\in\mathbb{N}$ and
	\begin{equation}\label{28}
	M^3_{\beta_j}<\left(M_{\alpha_{j}}+M_{\beta_j}\right)^3
	<M_{\alpha_{j}}+M_{\beta_{j+1}}
	<M_{\alpha_{j+1}}+M_{\beta_{j+1}}<2M_{\beta_{j+1}}.
	\end{equation}

	Setting 
	$$f:=\sum_{j=1}^{\infty}\psi_{M_{\beta_{j+1}}}P_j $$		
	we observe by \eqref{24} that this series converges in $L_p(G_m)$ norm. Hence $f \in L_p(G_m)$ and this	series is the Vilenkin-Fourier series of $f.$ Moreover, since the spectra of the polynomials $P_j$ are pairwise disjoint and
\begin{equation*}
\widehat{f}\left( k\right) =\left\{
\begin{array}{l}
\text{ }c_k^l,\text{ \ if \ } k=M_{\beta_{l+1}}+M_{\alpha_{l}} ,\ldots,M_{\beta_{l+1}}+M_{\alpha_{l+1}}-1, \ l\in \mathbb{N}_+, \\
\text{ }0,\text{ otherwise.}
\end{array}
\right.
\end{equation*}	
for any $j \in \mathbb{N}_+$ we have that

\begin{eqnarray} 
&& \ \ \sigma_{\left(M_{\alpha_{j+1}}+M_{\beta_{j+1}}\right)^3}f-\sigma_{M_{\alpha_j}+M_{\beta_{j+1}}}f\\ \notag
&=&\psi_{M_{\beta_{j+1}}}P_j-\sum_{k=M_{\alpha_j}+M_{\beta_{j+1}}}^{M_{\alpha_{j+1}}+M_{\beta_{j+1}}-1}\frac{k}{\left(M_{\alpha_{j+1}}+M_{\beta_{j+1}}\right)^3}c^j_k\psi_k\\ \notag
&+&\sum_{l=0}^{j-1}\sum_{k=M_{\alpha_l}+M_{\beta_{l+1}}}^{M_{\alpha_{l+1}}+M_{\beta_{l+1}}-1}\left(\frac{k}{M_{\alpha_{j}}+M_{\beta_{j+1}}}-\frac{k}{\left(M_{\alpha_{j+1}}+M_{\beta_{j+1}}\right)^3}\right)c^l_k\psi_k \\ \notag
&:=& I-II+III.
\end{eqnarray}

Since $\left\vert c^j_k\right\vert<C,$ for any $j\in \mathbb{N},$ according to \eqref{28} for $II$ we find that 
\begin{eqnarray*}
\vert II\vert&\leq& C\sum_{k=0}^{M_{\alpha_{j+1}}+M_{\beta_{j+1}}-1}\frac{k}{\left(M_{\alpha_{j+1}}+M_{\beta_{j+1}}\right)^3}\leq \frac{C\left(M_{\alpha_{j+1}}+M_{\beta_{j+1}}\right)^2}{\left(M_{\alpha_{j+1}}+M_{\beta_{j+1}}\right)^3}\\
&\leq&\frac{C}{\left(M_{\alpha_{j+1}}+M_{\beta_{j+1}}\right)}\to 0, \  \text{as}  \ j\to \infty.
\end{eqnarray*}

Analogously, if we apply again \eqref{28} we get that
\begin{eqnarray*}
&&\vert III \vert\leq \sum_{l=0}^{j-1}\sum_{k=M_{\alpha_l}+M_{\beta_{l+1}}}^{M_{\alpha_{l+1}}+M_{\beta_{l+1}}}\left(\frac{k}{M_{\alpha_{j}}+M_{\beta_{j+1}}}-\frac{k}{\left(M_{\alpha_{j+1}}+M_{\beta_{j+1}}\right)^3}\right)\vert c^l_k\vert\\
&\leq& \notag C\sum_{k=0}^{M_{\alpha_{j}}+M_{\beta_{j}}} \frac{k}{M_{\alpha_{j}}+M_{\beta_{j+1}}}
\leq C\sum_{k=0}^{2M_{\beta_{j}}} \frac{k}{M_{\beta_{j+1}}}\leq  \frac{C\left(M_{\beta_{j}}\right)^2}{M_{\beta_{j+1}}}\to 0, \  \text{as}  \  j\to \infty.
\end{eqnarray*}
	Since every $x \in E$ belongs to infinitely many of the sets $A_j$  it follows from \eqref{25} that $\vert I \vert=1, \ \text{for } \  x\in I_{A_{j}}, \ \text{for infinitely many } \  j.$  Hence, we obtain that  the Vilenkin-Fourier series of $f$ diverges at every point $x \in E.$ The proof is complete.
\end{proof}

\section{Corollaries and final remark}
\begin{remark}
In fact, the theorem \ref{Theorem_11} follows from the general result of Karagulyan \cite{Kar}, since the Vilenkin-Fejer means satisfy localization property, but we provide an alternative approach and the constructed function in our proof has a simple and explicit representation.
\end{remark}

\begin{corollary} \label{Theorem_111}(see Bugadze \cite{Bitsadze} and Goginava \cite{Gog})
	Let $1\leq p<\infty.$ If $E\subseteq G_m$ is a set of measure zero, then there exists a function $f\in L_p$ such that F\'ejer means with respect to Walsh systems diverge on this set. 
\end{corollary}


\begin{corollary} \label{Theorem_11111}(see \cite{PSTW} and \cite{sws})
	Let $1\leq p<\infty.$ If $E\subseteq G_m$ is a set of measure zero, then there exists a function $f\in L_p(G_m)$ such that partial sums with respect to Vilenkin (Walsh) systems diverge on this set. 
\end{corollary}

\begin{remark}
Almost everywhere convergence of some N\"orlund and $T$ means can be found in \cite{PTW2} (see also \cite{BNPT}).
We note that this type of construction in Theorem 	\ref{Theorem_11} can be used to prove similar results for N\"orlund and $T$ means.
\end{remark}
\textbf{Acknowledgement:} We thank Professor Grigori Karagulyan for some good advices, which has improved the final version of this paper.






\end{document}